%% file: MAIN.tex
\documentclass[12pt]{article}
\usepackage{amsmath}
\usepackage{amssymb}
\usepackage[english]{babel}
\usepackage{amsthm}
\usepackage[dvipdfmx]{graphicx}
\usepackage{tikz}
\usepackage{enumerate}
\usepackage{authblk}
\usepackage{hyperref}
\usepackage[textsize=footnotesize]{todonotes}
\newcommand \tp {{\mathrm {tp}}}
\newcommand \stp{{\mathrm {stp}}}
\newcommand \acl {{\mathrm {acl}}}

\newcommand \ind {\mathop{\smile \hskip -0.9em ^| \ }}

\newcommand \Th {{\mathrm {Th}}}

\newcommand \Z  {{\mathbb Z}}

\newcommand \Q {{\mathbb Q}}

\newcommand \M {{\cal M}}

\newtheorem{theorem}{Theorem}
\newtheorem{proposition}[theorem]{Proposition}
\newtheorem{lemma}[theorem]{Lemma}
\newtheorem{corollary}[theorem]{Corollary}
\newtheorem{fact}[theorem]{Fact}
\newtheorem{claim}{Claim}[theorem]
\newtheorem*{claim*}{Claim}
\theoremstyle{definition}
\newtheorem{definition}[theorem]{Definition}
\newtheorem{example}[theorem]{Example}

\theoremstyle{definition}
\newtheorem{remark}[theorem]{Remark}

\newtheorem{thmA}{Theorem}

\newcommand{\pair}[1]{\langle #1 \rangle}
\title{On chromatic number of countable graphs}

\author[1]{Hirotaka Kikyo}
\author[2]{Koitaro Nakaura}
\author[3]{Akito Tsuboi \thanks{This work was partially supported by JSPS KAKENHI Grant Numbers 21K03336 and 25K07096.}}
\affil[1]{Kobe University}
\affil[2]{University of Tokyo}
\affil[3]{University of Tsukuba}

\begin{document}

\maketitle
\begin{abstract}
    This paper investigates when countable graphs have a finite or an infinite chromatic number through model‑theoretic methods. For Fra\"{i}ss\'{e} limits, we show that instability forces the chromatic number to be infinite, yielding a complete classification of homogeneous graphs with a finite chromatic number. In contrast, Hrushovski construction always produces graphs of finite chromatic number, though the value can be made arbitrarily large. In tame settings—such as stable graphs of $U$‑rank one and graphs definable in o‑minimal structures—an infinite chromatic number necessarily yields arbitrarily large cliques. These results provide a unified framework connecting structural model‑theoretic properties with chromatic behavior.
\end{abstract}
\section{Introduction}

\input{introduction}

\section{Preliminaries and Basic Results}
We view a graph as a model-theoretic structure $G = (G, R)$, where $R \subset G^2$ is a symmetric and irreflexive relation representing the edge set. 
We refer to induced subgraphs simply as subgraphs, so that subgraphs coincide with substructures.
When $R=\{(a, b) \in G^2: a \neq b\}$, the graph $G$ is called complete.  
 A complete subgraph $H \subset G$ is called a clique of $G$.
 The degree of a vertex $v$ is the number of vertices adjacent to $v$.
The symbol \( G \) is used exclusively to denote a graph, typically assumed to be countably infinite.  
Its monster model, assumed to be sufficiently saturated, is denoted by \(\M \).
For two sets \( A \) and \( B \), the union \( A \cup B \) is sometimes written simply as \( AB \),  
provided no confusion arises.

\begin{definition}
Let $G$ be a graph. 
The chromatic number $\chi(G)$ of $G$ is the smallest cardinal $\kappa$ such that there exists a function (called a coloring) $f:G \to \kappa$ such that all adjacent vertices have a different color.    
\end{definition}

\begin{remark}\label{rem-first}
    \begin{enumerate}
    \item If $G$ has a clique of size $\kappa$, then $\chi(G) \geq \kappa$.
However, the converse does not hold in general, as illustrated in Proposition \ref{Kikyo}. 
        \item Let $G$ be an infinite graph with $\chi(G) = n <\omega$. Then
 \begin{itemize}
     \item[(i)] there is a finite subgraph \(F \subset G\) with \(\chi(F)=n\) (De Bruijn-Erd\H{o}s Theorem \cite{DBE}); and  
     \item[(ii)] for every \(H \equiv G\), we have \(\chi(H)=n\).
 \end{itemize}

 \begin{proof}
(i) Towards a contradiction, suppose that no finite subgraph of \(G\) has chromatic number \(n\).  
Let \(T\) be the theory of \(R\)-graphs, and let \(\mathrm{Diag}(G)\) be the diagram of \(G\).  
Expand the language by a new binary relation symbol \(E\), and let \(T^{*}\) be the theory $T \cup \mathrm{Diag}(G)$ plus the following sentence: 
\begin{itemize}
    \item $E(x,y)$ is an equivalence relation with \(n-1\) classes, each of which is \(R\)-free.
\end{itemize}
By our assumption, \(T^{*}\) is consistent: Let \(G^{*} \models T^{*}\).  
Clearly \(\chi(G^{*}) < n\).  
Since \(G^{*} \models \mathrm{Diag}(G)\), we may view \(G\) as an induced subgraph of \(G^{*}\).  
But then \(\chi(G) \le \chi(G^{*}) < n\), contradicting \(\chi(G)=n\).  
This proves (i).
\medbreak
(ii) follows directly from (i), since the property “there exists a finite subgraph of chromatic number \(n\)” is first‑order expressible. 
Alternatively, one may argue directly as follows.  
Let \(T^{*}\) be the theory \(\Th(G)\) together with the sentence asserting that \(E(x,y)\) is an equivalence relation with \(n\) classes, each \(R\)-free.  
This theory is consistent by the same finite‑satisfiability argument as above.  
Let \(G^{*} \models T^{*}\) be saturated.  
Any \(H \equiv G\) elementarily embeds into \(G^{*}\), and since \(\chi(G^{*}) \le n\), we obtain \(\chi(H) \le n\).  
By symmetry with the previous paragraph, \(\chi(H)=n\).
\end{proof}
       \item Let \( G \) be an infinite graph with \( \chi(G) = \omega \), and suppose \( G = H_0 \cup \dots \cup H_{n-1} \) is a finite partition.  
    Then there exists \( i < n \) such that \( \chi(H_i) = \omega \).
\item \label{edge-set-partition}
    Let \( (G, R) \) be a graph, and suppose \( R = C_0 \cup \dots \cup C_{n-1} \) is a finite partition of the edge set,  
where each \( C_i \subset R \) is symmetric.  
If the chromatic number \( \chi(G, R) \) is infinite, then there exists some \( i < n \) such that \( \chi(G, C_i) \) is infinite.
\end{enumerate}
\end{remark}
It is evident that a triangle-free graph cannot contain an infinite clique. 
Nevertheless, it is known that the triangle-free random graph $G$ satisfies 
$\chi(G) \geq \omega$. Proposition {\ref{Kikyo}} below extends this observation.
A more general result will be established later in Section~3.1.

\input{fraisselimit}

\begin{definition}[Mycielskian]
\label{def:mycielskian}
 Let $A$ be a finite graph. For each vertex $v$ of $A$,
let $A_v = A \oplus_{A\setminus\{v\}} A'$ where $A'$ is a graph isomorphic to
$A$ over $A\setminus\{v\}$. The element in $A'\setminus A$
is called a sibling of $v$.
Let $\tilde{A}$ be the free amalgam of all $A_v$ with $v \in A$ over $A$.
Let $S$ be the set of all siblings in $\tilde{A}$ of vertices in $A$.
Note that $S$ is an edge-free set in $\tilde{A}$.
Finally,  let $A_M = \tilde{A} \cup \{v^*\}$ be a
graph with a new vertex $v^*$ such that 
the set of all neighbors of $v^*$ in $A_M$ is $S$.
$A_M$ is called the \emph{Mycielskian} of $A$.

We can represent $A_M = \tilde{A} \oplus_S (Sv^*)$
and $Sv^*$ is the free amalgam of all $uv^*$ with $u \in S$
over $v^*$.
Hence, the Mycielskian of a graph is obtained by free amalgamation construction.
\end{definition}


\begin{fact}[\cite{M}]
Let $A_M$ be the Mycielskian of $A$.
Then $\chi(A_M) = \chi(A)+1$.
\end{fact}

\begin{proof}
    Consider graphs in Definition {\ref{def:mycielskian}}. By the construction, we can easily obtain $\chi(A_M) \le \chi(A) +1$.

We show $\chi(A_M) \ge \chi(A)+1$.
    Suppose that $g\colon A_M\to \chi(A_M)$ is an optimal coloring. Let $i$ be the color of $v^*$, and let $X=\{v\in A\mid g(v) =i\}$. Note that there is no edge in $X$. Let $X'$ be the set of all siblings of vertices in $X$. By the choice of $i$, $g(x')\ne i$ holds for any $x'\in X'$. Consider the graph $(A\backslash X)\cup X'$, which is isomorphic to $A$. The graph is painted with the colors $\chi(A_M)\backslash \{i\}$, and hence we have $\chi(A) \le \chi(A_M) -1$. 
\end{proof}


\begin{proposition}\label{Kikyo}
Suppose $K$ is a class of finite graphs with the free amalgamation property (FAP),
and let M be the Fraïssé limit of $K$.
Then $\chi(M) = \omega$. 
\end{proposition}

\begin{proof}
    For a contradiction, suppose that $\chi(M)=n^* < \omega$.
        There exists a finite graph $G \subset M$ with $\chi(G)=n^*$ 
        by Remark {\ref{rem-first}}, 2 (i).
        Since \(G \in K\) and $K$ has FAP, 
        the Mycielskian $G'$ of $G$ belongs to $K$.  Hence, $G'$ is isomorphic to a subgraph of $M$.  But 
        $\chi(G') > \chi(G) = n^* = \chi(M)$. This is a contradiction.
%
\end{proof}
\begin{remark}
Let $G$ be the Fra\"{i}ss\'{e} limit of $K$, where $K$ has FAP.
Then $\Th(G)$ has the independence property and is therefore unstable. 
This can be shown as follows: 
Let $a,b,c \in G$ be such that $R(a,b)$ and $\neg R(a,c)$. By FAP, we can assume that $b$ and $c$ are $R$-free over $a$.
By FAP again, there is an $R$-free indiscernible sequence $\{b_ic_i\}_{i \in \omega}$ over $a$ starting with $b_0c_0=bc$.
Let $I=\{d_i\}_{i \in \omega}$ be a sequence defined by $d_{2i}=b_i$ and $d_{2i+1}=c_i$. It is an indiscernible set.
For any partition $X \cup Y$ of $\omega$, there is $a_{X,Y}$ such that $\bigwedge_{i \in X}R(a_{X,Y},d_i) \wedge \bigwedge_{j \in Y} \neg R(a_{X,Y},d_j)$. 
This means that $T$ has the independence property. 
\end{remark}

\begin{remark}\label{rem-order-prop}
\begin{enumerate}
    \item Suppose that $\bigvee_{i<n} \varphi_i(\bar x, \bar y)$ has the order property. 
    Then, some $\varphi_i(\bar x, \bar y)$ has the order property. 
    \item Suppose that $\varphi(\bar x,\bar y)$ has the order property. Then $\neg \varphi(\bar x,\bar y)$ also has the order property. 
    \item From 1 and 2, we see the following: 
    Suppose that $\bigvee_{i<n} \bigwedge_{j<n_i} \varphi_{ij}(\bar x, \bar y)$ has the order property. Then, some $\varphi_{ij}(\bar x,\bar y)$ has the order property. In particular, if $T=\Th(G)$ is unstable and has QE, then the edge relation $R(x,y)$ has the order property. 
\end{enumerate}
\end{remark}

The following lemma is implicit in the proof of Proposition 3.11 in \cite{HKS}. 
Recall that the strong type $\stp(a/A)$ is the set of all formulas of the form $E(x,a)$, where $E$ is an $A$-definable finite equivalence relation. Equivalently, $\stp(a/A)=\tp(a/\acl^{eq}(A))$. Strong types are stationary. 
We write $a \ind_A b$ if $a$ and $b$ are independent over $A$.
For further properties of strong types and independence, see \cite{PillayStability} or \cite{ShelahClassification}. 

\input{strong-type.tex}

\section{Fra\"{i}ss\'{e} Limits}

\subsection{Infinite chromatic numbers}

Recall that $K$ denotes a non-trivial infinite class of finite graphs having the hereditary property.

\input{Fraisse-infinite-chromatic.tex}

\begin{theorem}\label{unstable,omega}
Let $G$ be the Fra\"{i}ss\'{e} limit of $K$.
Suppose that the theory of $G$ is unstable. 
Then $\chi(G)=\omega$.
\end{theorem}
\begin{proof}
  Assume that $\chi(G) = n \in \omega$ towards a contradiction.
By instability and quantifier elimination, it follows that $R$ has the order property.
  So, there exist sequences $\{a_q\}_{q \in \Q}$ and $\{b_r\}_{r \in \Q}$ (in the monster model) such that $R(a_q,b_r)$ if and only if $q < r$. 
  Since $\Th(G)$ is $\omega$-categorical, we can assume that $\{a_q\}_{q \in \Q}$ and $\{b_r\}_{r \in \Q}$
  are sequences in $G$.
   By the finiteness of $\chi(G)$, we can assume that $\{a_q\}_{q \in \Q}$ and $\{b_r\}_{r \in \Q}$ are $R$-free sequences.
   Consider $I=\{a_q\}_{q \in (0,1)}$, $J_0=\{b_r\}_{r \in (-\infty,0)}$ and $J_1=\{b_r\}_{r \in (1,+\infty)}$. 
   Then, $I \cup J_0$ is an $R$-free sequence, while all elements in $I=\{a_q\}_{a \in (0,1)}$ are $R$-adjacent to all elements in $J_1$. 
   From this observation, using elimination of quantifiers, we have the following:
\begin{claim}
    Let $J$ be an infinite $R$-free set. Let $J=J_0 \cup J_1$ be a partition. Then, there is an infinite $R$-free set $I$ such that $I \cup J_0$ is $R$-free, and that all elements in $I$ is $R$-adjacent to all elements in $J_1$. 
\end{claim}
    Let $P_0 \cup P_1 \cup \dots \cup P_{n-1}$ be a partition of $G$ into vertex color classes. 
\begin{claim}
    There exist infinite subsets $A_i \subset P_i$ for each $i<n$ such that $\bigcup_{i<n} A_i$ is $R$-free.
\end{claim}
Suppose that this is not the case and choose a maximal set $F \subset n$ such that there exist infinite $A_i \subset P_i$ $(i \in F)$ with $\bigcup_{i \in F} A_i$ being $R$-free. 
For each $i \in F$, let $A_i=A_i^0 \cup A_i^1$ be a partition of \(A_i\) into two infinite subsets, and define $J_0= \bigcup_{i \in F}A_i^0$ and $J_1=\bigcup_{i \in F} A_i^1$. 
Using Claim A, we can find an infinite set $I$ such that (i) $I \cup \bigcup_{i \in F}J_0$ is $R$-free, and (ii) all elements in $I$ are $R$-adjacent to all elements in $J_1$. 
By (ii), $I \cap \bigcup_{i \in F} P_i$ is empty. 
Since $I$ is infinite, taking a subset of $I$, we can assume $I$ is contained in $P_k$ for some $k \notin F$.
But then, by (i), $F \cup \{k\}$ contradicts the maximality of $F$. 
(End of Proof of Claim B)
\medbreak
Now, consider $A:=\bigcup_{i<n} A_i$. 
$A$ is $R$-free, so there is $a$ such that $a$ is adjacent to every element in $A$. So, $a$ does not belong to $\bigcup_{i<n} P_i$. A contradiction. 
\end{proof}

\subsection{Finite Chromatic Numbers}
\begin{theorem}
    Let $G$ be a countable homogeneous graph (a graph constructed by Fra\"{\i}ss\'e construction). 
    Suppose $\chi(G)=n \in \omega$. 
    Then $G$ is isomorphic to one of the following graphs:
    \begin{itemize}
        \item The free amalgam of infinitely many copies of $K_n$, the complete graph with $n$ vertices;
        \item $K_{\omega}^n$, the complete \(n\)-partite graph in which every partition class has exactly $\omega$ vertices.
    \end{itemize}
\end{theorem}

\begin{proof}
    Observe that $T=\Th(G)$ is stable. This follows from the fact that an unstable homogeneous graph must have infinite chromatic number (see Theorem \ref{unstable,omega}). 
    We will use this stability condition later in the argument.
    First, we consider the case where every node has finite degree.
    By homogeneity, there exists $m \in \omega$ such that $\deg(a)=m$ for all $a \in G$. 
    By $\omega$-categoricity, $\acl(a)$ is finite for all $a \in G$. Since the connected component $C(a)$ of $a$ is contained in $\acl(a)$, $C(a)$ is also a finite set. 
    \begin{claim}
         $C(a)$ is isomorphic to $K_n$.   
    \end{claim}
    First, we show that $C(a)$ is a complete graph. 
    Suppose, for a contradiction, that there exist distinct $d_0,d_1 \in C(a)$ such that no edge directly connects \( d_0 \) and \( d_1 \). 
    Choose $e_0$ and $e_1$ such that $C(e_0) \cap C(e_1)=\emptyset$. 
    The quantifier-free types of the pairs \( (d_0, d_1) \) and \( (e_0, e_1) \) are the same.
    However, there is no automorphism sending \( d_0, d_1 \) to \( e_0, e_1 \),  as the former lie in the same connected component while the latter do not.  This contradicts the homogeneity of the structure.
Since \( \chi(G) = n \), the cardinality of \( C(a) \) must be \( n \).
    (End of Proof of Claim A)
\medbreak
By Claim A, we conclude that $G$ is the free amalgam of infinitely many copies of $K_n$. 
\medbreak
Now we consider the case where every node has infinite degree. 
Let $a \in G$, and choose distinct elements $b_i$ $(i \in \omega)$ such that each pair $(a, b_i)$ forms an $R$-edge. 
By Ramsey's theorem, we can assume that the sequence $B=\{b_i\}_{i \in \omega}$ forms an $R$-free indiscernible sequence. 
Since a Morley sequence over \( \emptyset \) always exists and must be \( R \)-free,  
it follows by homogeneity that \( B \) is itself a Morley sequence.
So, $a$ and $b:=b_0$ are independent over $\emptyset$, and $R(a,b)$ holds. 
By Lemma~\ref{lemma-stp}, if \( \stp(a) = \stp(b) \), then a contradiction arises easily.  
Therefore, there must exist a finite equivalence relation \( E(x, y) \) such that \( \neg E(a, b) \) holds.
By quantifier elimination (QE), it is easy to see that \( E(x, y) \) is equivalent to a disjunction of formulas drawn from the following list:
\[
x = y, \quad R(x, y), \quad x \neq y \wedge \neg R(x, y).
\]
 The formula \( x = y \) must appear in the disjunction, since \( E(x, y) \) is reflexive.  
However, if all three formulas appear in the disjunction,  
then \( E(x, y) \) becomes the trivial equivalence relation that holds for all pairs \( x, y \in G \).  
This contradicts the assumption that \( \neg E(a, b) \) holds. 
So, by a simple calculation, we see that the remaining possibilities for $E$ are 
either  
\[
E_0(x, y) \equiv (x = y \vee R(x, y)) \quad \text{or} \quad E_1(x,y) \equiv\neg R(x, y).
\]
 \( E=E_0\) is impossible.  
Indeed, by our assumption \( E_0(a, b_i) \) holds for all \( i \in \omega \), so each \( b_i \) lies in the same \( E_0 \)-class as \( a \).  
Since \( E_0 \) is transitive, it follows that \( E_0(b_i, b_j) \) must hold for all distinct \( i, j \),  
which implies \( R(b_i, b_j) \) for all \( i \neq j \).
This contradicts the assumption that the sequence \( \{b_i\}_{i \in \omega} \) is \( R \)-free.  
Thus, we must have \( E = E_1 \).  
In this case, each \( E_1 \)-class is \( R \)-free, and any two elements from distinct classes are adjacent.  
It follows that the number of equivalence classes is exactly \( n \), and each class is infinite.  
Hence, \( G \) is isomorphic to the complete \( n \)-partite graph.   
\end{proof}

\section{Hrushovski Construction}
Let $\alpha$ be a real number such that $0\leq \alpha \leq 1$. 
For a finite graph $F$, define $\delta_\alpha(F)$ as the number $|F|-\alpha \cdot (\text{the number of edges in $F$}$).
Let $K_\alpha$ be the class of all finite graphs $F$ such that every subgraph $F_0 \subset F$ satisfies $\delta_\alpha(F_0) \geq 0$.  
Let $A \subset B$ be graphs with $A$ finite.
We write $A \leq_\alpha B$ if, 
whenever 
$A \subseteq C \subseteq  B$ with $C$ finite, 
it follows that $\delta_\alpha(A) \leq \delta_\alpha(C)$.
We write $A <_\alpha B$ if, 
whenever 
$A \subsetneq C \subseteq B$ with $C$ finite, 
it follows that $\delta_\alpha(A) < \delta_\alpha(C)$.
$A \leq_\alpha B$ and $A <_\alpha B$ are equivalent if $\alpha$ is an irrational
number, but they are different if $\alpha$ is a rational number.

Let $K$ be a subclass of $K_\alpha$.
Let $\leq$ denote $\leq_\alpha$ or $<_\alpha$.
We will assume $\alpha < 1$ when we work with $<_\alpha$.

$(K, \leq)$  has the \emph{amalgamation property} (AP)
if whenever $A \leq B$, $A \leq C$ and $A, B, C \in K$ 
there is a graph $D \in K$, graph embeddings $f_B: B \to  D$ and
$f_C: C \to D$ such that 
$f_B(B) \leq D$, $f_C(C) \leq D$ and $f_B|A = f_C|A$. 
$(K, \leq)$ has the \emph{free amalgamation property} (FAP)
if whenever $D = BC$ is the free amalgam of $B$ and $C$ over $A = B\cap C$
with $A \leq B$ and $A \leq C$ then $D$ belongs to $K$.
If $(K, \leq)$ has FAP then it has AP.

A graph $M$ is a \emph{generic structure} for $(K, \leq)$
if (i) $M$ is countable, (ii) any finite substructures of $M$ belong to $K$,
(iii) for any finite substructures $A$ of $M$ there is a finite 
substructure $B$ of $M$ satisfying $A \subseteq B \leq M$,
(iv) for any finite substructures $A \leq M$ and for any structures $B \in K$
with $A \leq B$ there is a graph embedding $f: B \to M$ such that
$f(B) \leq M$ and $f(x) = x$ for any $x \in A$.
If the coefficient $\alpha$ of the predimension function is 0 then a generic structure for $(K, \leq)$ is 
just the Fra\"{\i}ss\'e limit of $K$.

We call $(K, \leq)$ an \emph{amalgamation class} if
$K$ contains an empty graph, 
$K$ is closed under isomorphic images
and substructures, and $(K, \leq)$ has AP.
Recall that we are assuming $K$ is non-trivial (i.e., $K$ contains a graph with an edge) from the preliminaries (Section 2).
Any amalgamation class has a unique generic structure up to isomorphism.
Most of the known amalgamation classes have FAP. 

%

\begin{proposition}
\label{prop:upperbound}
Let $\alpha$ be a real number with $0 < \alpha \leq 1$ and
$(K, \leq)$ an amalgamation class where 
$\leq$ denotes $\leq_\alpha$ or $<_\alpha$.
Let $k^*$ be a natural number with $\alpha k^*  > 2$.
Let $G$ be the generic structure for $(K, \leq)$.
Then $\chi(G) \leq k^*$. 
\end{proposition}

\begin{proof} By Remark {\ref{rem-first}}, 2 (i),
it is enough to show that any finite substructure of $G$
can be $k^*$-colored.

 By induction on $n \in \omega$, we show that every graph $A$ in $K$ 
with $|A| = n$ has a valid coloring using $k^*$ colors. 
This is trivial if $|A| \leq k^*$.

Let $A \in K$ with $|A|=n > k^*$. 
   \begin{claim}
       There exists a vertex $a \in A$ with degree less than $k^*$ in $A$.
   \end{claim}
   Suppose otherwise. Then every vertex in $A$ has degree 
at least $k^*$.
So, there are at least $k^*n/2$ edges in $A$. 
Since $\alpha k^* > 2$,
we have $\delta(A)=n-\alpha\cdot k^* n/2 <0$, leading to a contradiction. (End of Proof of Claim A)
   \medbreak
   Choose $a \in A$ as described in the above claim.
   By the induction hypothesis, there exists a valid coloring $g:A \setminus \{a\} \to k^*$.
   Let $N \subset A \setminus \{a\}$ be the set of vertices adjacent to $a$ in $A$. We have $|N| < k^*$ by the choice of $a$. We can choose a color $k \in k^*$
such that $k \neq g(v)$ for any $v \in N$.
   Then, $f=g \cup \{(a, k)\}$ is a valid coloring of $A$ with $k^*$ colors. 
   Thus, any $A \in K$  with $|A| = n$ is $k^*$-colored.
\end{proof}

\begin{corollary}
Let $\alpha$ be a real number with $2/3 < \alpha \leq 1$ and
$(K, \leq)$ an amalgamation class where 
$\leq$ denotes $\leq_\alpha$ or $<_\alpha$.
Let $G$ be the generic structure for $(K, \leq)$.
Assume that $K$ contains arbitrarily long paths.
Then $\chi(G) = 3$.
\end{corollary}

\begin{proof}
We have $3 \alpha > 2$ because $\alpha > 2/3$. By Proposition {\ref{prop:upperbound}}, $\chi(G) \leq 3$.
Choose sufficiently long paths $L_1$ and $L_2$.
Let $v_1, u_1$ be both ends of $L_1$ and
$v_2, u_2$ be both ends of $L_2$.
Since $L_1$ and $L_2$ are long, we can assume that $\{v_1, u_1\} \leq L_1$
and $\{v_2, u_2\} \leq L_2$. Also, we can assume that $L_1$ has 
an odd length and $L_2$ has an even length.
Since $\{v_1, u_1\}$ and $\{v_2, u_2\}$ are isomorphic graphs,
there is a graph $D$ in $K$ such that there are embeddings $f_1: L_1 \to D$
and $f_2: L_2 \to D$ satisfying $f_1(v_1) = f_2(v_2)$ and 
$f_1(u_1) = f_2(u_2)$.
Then $f_1(L_1) \cup f_2(L_2)$ must contain a cycle of odd length.
Hence, $G$ must have a cycle of an odd length as a substructure.
Since any cycle of an odd length has chromatic number 3,
we have $3 \leq \chi(G)$. Therefore, $\chi(G) = 3$.
\end{proof}

Assuming FAP on $(K, \leq)$, we present some result
on lower bounds of $\chi(G)$.

\begin{lemma}
\label{lem:closedness}
 Let $A$ be a finite graph. Suppose $\alpha < 1/\varDelta(A)$
where $\varDelta(A)$ denotes the maximum degree of vertices in $A$.
Then $B <_{\alpha} A$ (hence $B \leq_{\alpha} A$) 
for any substructure $B$ of $A$.
\end{lemma}

\begin{proof}
 We can represent $B = B_0 \subset B_1 \subset \cdots \subset B_k = A$
where each $B_i$ is a substructure of $A$ and 
$B_{i+1} = B_i \cup \{v_i\}$ for each $i = 0$, $\ldots$, $k-1$.
Since the number of edges between $v_i$ and $B_{i-1}$ is at most 
$\varDelta(A)$, we have $B_i <_\alpha B_{i+1}$ for
$i = 0$, $\ldots$, $k-1$.
Hence $B <_{\alpha} A$.
\end{proof}

%
%
%

Recall that the Mycielskian of a graph is obtained by free amalgamation construction.
With Lemma {\ref{lem:closedness}}, we have the following:

\begin{lemma}
Let $\alpha$ be a real number with $0 \leq \alpha \leq 1$ and
$(K, \leq)$ an amalgamation class with FAP where 
$\leq$ denotes $\leq_\alpha$ or $<_\alpha$.
Suppose $A\in K$ and $A'$ is the Mycielskian of $A$.
If $\alpha < 1/\varDelta(A')$ then $A'$ belongs to $K$.
\end{lemma}


Now, we have the following proposition.

\begin{proposition}
\label{prop:lowerbound}
Let $\alpha$ be a real number with $0 < \alpha \leq 1$ and
$(K, \leq)$ an amalgamation class with FAP where 
$\leq$ denotes $\leq_\alpha$ or $<_\alpha$.
Let $G$ be the generic structure for $(K, \leq)$.
For any integer $n > 0$ there is $\varepsilon > 0$ such that
if $0 < \alpha < \varepsilon$ then $\chi(G) \geq n$.
\end{proposition}

\section{Graphs without complex structures}
As stated, our goal is to show that an infinite chromatic number implies the existence of an infinite clique, provided the edge structure is sufficiently simple.
So far, we have established this implication in the following two cases.

\subsection{$U$-rank one graphs}
\begin{theorem}
    Let \( G \) be a graph whose theory has \( U \)-rank one, meaning that every element in the monster model has \( U \)-rank at most 1.  
Assume \( \chi(G) \geq \omega \). Then \( G \) must contain an infinite clique.
\end{theorem}

\begin{proof}
By assuming that $G$ does not contain an infinite clique, we derive a contradiction. 
    Let $\Gamma(x,y)$ be the following set of formulas:
    \[
  \{(\exists^{< n} u) \varphi(u,y) \to \neg \varphi(x,y): n \in \omega, \ \varphi(x,y) \text{ a formula}\}.
    \]
\begin{claim}
    $\Delta(x,y):=$` $\stp(x)=\stp(y)$' $ \cup \, \{R(x,y)\} \cup \Gamma(x,y) $ is inconsistent. 
\end{claim}
Suppose, for a contradiction, that $\Delta(x,y)$ is consistent, and let
$(a, b)$ be a realization.
Since $R$ is irreflexive, the condition $\stp(x) = \stp(y)$ implies that both $a$ and $b$ are non-algebraic.
Hence, $U(a) = U(b) = 1$.
Since $(a, b)$ satisfies $\Gamma$, we must have $a \ind b$. 
By Lemma \ref{lemma-stp}, this implies the existence of an infinite clique. Hence, $\Delta$ must be inconsistent. 
(End of Proof of Claim A)
\medbreak
Since \( \Delta \) is inconsistent, there exist a finite equivalence relation \( E(x, y) \), an integer \( n \in \omega \), and finitely many formulas \( \varphi_i(x, y) \) for \( i < k \), such that the following sentence holds:
\[
(\forall x \forall y )\left[ E(x, y) \wedge R(x, y) \rightarrow \bigvee_{i < k} \left( (\exists^{< n} u)\, \varphi_i(u, y) \wedge \varphi_i(x, y) \right) \right].
\]
By Remark \ref{rem-first}, since there are only finitely many $E$-classes, one of them—say $C$—must have infinite chromatic number. 
However, by the sentence above, every element \( a \in C \) has finite degree; in fact, we have \( \deg(a) \leq kn \).  
But the uniform finiteness of degrees implies that \( \chi(C) \) is finite—a contradiction.
\end{proof}

\begin{example}
Let \(G = \bigoplus_{n \in \omega} K_n\), the disjoint union of finite complete graphs.  
Then \(\chi(G)\) is infinite, although \(G\) contains no infinite clique.  
In the monster model, however, an infinite clique does appear.  
Note also that the theory has \(U\)-rank \(2\), since the relation
\(
E(x,y) := R(x,y) \vee x=y
\)
is an equivalence relation with infinitely many classes.
\end{example}
\subsection{o-minimal graphs} 

\input{o-minimal-graph}

\end{document}

%% file: introduction.tex
Let \(G = (G, R)\) be a graph, where \(R \subset G^2\) is a symmetric and irreflexive relation representing the edges of \(G\).  
For a cardinal \(\kappa\), a (vertex) coloring of \(G\) with \(\kappa\) colors is a function \(f : G \to \kappa\) such that  
\(f(u) \neq f(v)\) for every edge \((u,v) \in R\).  
The chromatic number \(\chi(G)\) is the least cardinal \(\kappa\) for which such a coloring exists.

If $G$ has a complete induced subgraph of size $\kappa$, then clearly $\chi(G) \geq \kappa$. However, the converse is not true in general.
In \cite{HKS}, Halevi, Kaplan and Shelah studied uncountable graphs with stability and showed that under additional conditions, the converse holds for these graphs. They also studied uncountable graphs with simple theories and demonstrated a weak converse.

In this paper, we investigate the chromatic number of countable graphs, focusing on those with relatively simple structure—such as those defined by Fraïssé construction and related model-theoretic techniques.
Our aim is to identify criteria specific to these classes of graphs that determine whether the chromatic number $\chi(G)$ is finite or infinite.
We also aim to show that for graphs with relatively simple structure,
if such a graph $G$ has an infinite chromatic number then it must contain arbitrarily large finite complete subgraphs.

In Section 2, we review basic definitions and facts concerning chromatic numbers.  
These facts are accompanied by brief proofs.  
In what follows, we outline the main contributions of this paper.
In Section 3, we examine graphs constructed as Fraïssé limits.
The following are simple but important observations.
\begin{itemize}
\item Let $G$ be the Fraïssé limit of a (non-trivial) class $K$ of finite graphs with the free amalgamation property.
Then $\Th(G)$ has the independence property, and is therefore unstable. \end{itemize}

In relation to this remark, we have the following theorem: 

\begin{thmA}
    Let $G$ be a countable homogeneous graph. If $G$ is unstable, then $\chi(G)$ is infinite. 
\end{thmA}
Combining this result with the preceding observation, we deduce that the triangle‑free random graph has infinite chromatic number, even though it contains no clique of size greater than $3$.
Concerning the case with finite chromatic number, we have the following:

\begin{thmA} Let $G$ be the Fraïssé limit of an infinite class of finite graphs.
If $\chi(G) = n \in \omega$ then $G$ is isomorphic either to the free amalgam of infinitely many copies of the complete graph $K_n$ with $n$ vertices
or to $K_\omega^n$, the complete $n$-partite graph in which every partition class has exactly $\omega$ vertices.
\end{thmA}

In Section 4, we study graphs constructed using Hrushovski method.  
Although Hrushovski construction is analogous in spirit to the Fraïssé limit,  
it introduces additional complexity through the use of 
predimension functions. 
Nevertheless, the number of chromatic configurations that can arise in such structures is finite.  
We establish the following result:
\begin{thmA}
Let $G$ be a generic structure (a binary graph) obtained by 
Hrushovski construction using a predimension function. 
Then $\chi(G)$ is finite.
Depending on the coefficient of the predimension function,
$\chi(G)$ can be arbitrarily large.
\end{thmA}
\medbreak
The following theorem shows that if $G$ has an infinite chromatic number and lacks structural complexity,
then the monster model contains an infinite clique. 

\begin{thmA}
    Let  $G$ be a countable graph with an infinite chromatic number. 
    If the structure of $G$ is rather simple in the following sense, then $G$ has arbitrarily large finite complete subgraphs. 
    \begin{enumerate}
        \item $G$ is stable of $U$-rank one. (In this case, $G$ has an infinite clique.)
        \item $G$ is definable in an o-minimal structure.
    \end{enumerate}
\end{thmA}

%% file: fraisselimit.tex
Let $K$ be an infinite class of finite graphs closed under graph isomorphisms.
We define some notions as follows:
$K$ has the \emph{hereditary property} (HP) if,
whenever $A \in K$ and $B$ is a substructure of $A$ then $B \in K$.
$K$ has the \emph{amalgamation property} (AP) if
whenever $A, B, C \in K$, $A$ is a substructure of both $B$ and $C$
then there is a graph $D \in K$ such that there are graph embeddings
$f: B \to D$ and $g: C \to D$ satisfying $f(x) = g(x)$ for any $x \in A$.
$K$ has the \emph{joint embedding property} (JEP) if
whenever $A, B \in K$ 
 there is a graph $C \in K$ such that there are graph embeddings
$f: A \to C$ and $g: B \to C$.
A graph $M$ is the \emph{Fra\"{\i}ss\'e limit} of $K$
if (i) $M$ is countable, (ii) any finite substructures of $M$ belong to $K$,
(iii) for any finite substructures $A$ of $M$ and for any structures $B \in K$
with $A \subset B$ there is a graph embedding $f: B \to M$ such that
$f(x) = x$ for any $x \in A$.

It is known that the Fra\"{\i}ss\'e limit of $K$ exists if and only if
$K$ has HP, AP, and JEP. Let $M$ be the Fra\"{\i}ss\'e limit of $K$.
$M$ is \emph{homogeneous} in the following sense:
Any graph isomorphisms between finite substructures of $M$ can be extended to a graph automorphism of $M$. Therefore, the theory of $M$ 
is countably categorical and has QE.

Let $D$ be a graph, and let $A$, $B$, and $C$ be subgraphs of $D$
with $A \subset B$ and $A \subset C$.
$D$ is called a \emph{free amalgam} of $B$ and $C$ over
$A$, written $D = B\oplus_A C$, if $D = B\cup C$ 
and $B \cap C = A$ as vertex sets, and $R^D = R^B \cup R^C$.  A class $K$ has the \emph{free amalgamation property} (FAP) if whenever $D = B \oplus_A C$ with 
$A, B, C \in K$ then $D \in K$.
It is obvious that if $K$ has FAP then $K$ has AP.

In what follows, $K$ denotes an infinite class of finite graphs, and we assume that it has the hereditary property and it is non‑trivial—that is, $ K$ contains at least one graph with an edge.

%% file: strong-type.tex
\begin{lemma} \label{lemma-stp}
    Let \( G \) be an infinite graph whose theory \( T = \operatorname{Th}(G) \) is stable, and let \( A \subseteq G \).  
Suppose there exists a pair $(a, b)$ of elements in the monster model $\M  \succ G$ such that $R(a, b)$ holds, $\stp(a / A) = \stp(b / A)$, and $a \ind_A b$. 
Then \( G \) contains an infinite clique.
\end{lemma}
\begin{proof}
As the proof remains unchanged under the assumption $A = \emptyset$, we proceed with this simplification. 
By moving $ab$ by an elementary mapping, we can assume $ab \ind G$. 
Hence, both non-algebraic types $\tp(a/G)$ and $\tp(b/G)$ are non-forking extensions of the strong type $\stp(a)$. 
Thus, we have $\tp(a/G)=\tp(b/G)$; call the type $p$. 
Also, notice the independence relation \(a \ind_G b\), which implies that $\tp(a/Gb)$ is a coheir of $\tp(a/G)$ by stability. 


We inductively choose elements $a_i \in G$ such that $\{a,b, a_0, \dots, a_i\}$ forms a clique.
By $a \ind_G b$ and $R(a,b)$, we have $a_0 \in G$ such that $R(a_0,b)$. 
Since $\tp(a/G)=\tp(b/G)$, we have $R(a_0, a)$. 
This completes the first step. 
Suppose that $a_0,\dots,a_{n-1}$ have been properly selected. 
Since $\{a,b, a_0, \dots, a_{n-1}\}$ is a clique, there is some $a_n \in G$ such that $R(a_n,b)$ and $R(a_n, a_i)$ for any $i<n$ by $a \ind_G b$. 
Also, we have $R(a_n, a)$ by $\tp(a/G)=\tp(b/G)$. 
Since $\{a,b, a_0, \dots, a_{n-1}\}$ forms a clique, it follows that $\{a,b, a_0, \dots, a_n\}$ also forms a clique.  
Hence, the element $a_n$ has been successfully constructed, which completes our proof.
\end{proof}

%% file: Fraisse-infinite-chromatic.tex
\begin{proposition} 
   Suppose \( K \) has FAP.
    Let $G$ be the Fra\"{i}ss\'{e} limit of \( K \).
Then for any infinite cardinal \( \kappa \), there exists a graph \( H \models \Th(G) \) of size \( \kappa \) such that \( \chi(H) \leq \omega \).
\end{proposition}

\begin{proof}
We work in the monster model $\M \succ G$. Choose an $R$-free set $H_0 \subset \M$ of size $\kappa$. Such an $H_0$ exists by the free amalgamation property. 
   Let $f_0:H_0 \to \omega$ be the coloring function such that $f_0(a)=0$ for all $a \in H_0$. 
Starting from $H_0$ and $f_0$, we define increasing sequences $\{H_n\}_{n \in \omega}$ of subgraphs of $\M$ and $\{f_n\}_{n \in \omega}$ of colorings such that 
 \begin{itemize}
     \item $|H_n|=\kappa$;
     \item $f_n:H_n \to \omega$ is a valid coloring;
     \item For every existential formula of the form  
\(\exists y \, \theta(a_1,\dots,a_k,y)\), where \(\theta\) is quantifier-free and  
\(a_1,\dots,a_k \in H_n\), if the formula holds in \(\M\), then there exists  
\(b \in H_{n+1}\) such that \(\theta(a_1,\dots,a_k,b)\) holds.
 \end{itemize}
If this construction is successfully carried out, then
\[
   H := \bigcup_{n \in \omega} H_n
\]
is an elementary submodel of \(\M\).  
Indeed, since the theory \(T\) admits quantifier elimination,  
the Tarski–Vaught test reduces to checking quantifier-free formulas,  and by construction \(H\) satisfies this condition.
Furthermore, $f=\bigcup_{n \in \omega}f_n:H \to \omega$ provides a valid coloring of $H$.
So, suppose that $H_n$ and $f_n$ were obtained properly. 
Let $\{\exists y_i \theta_i(\bar a_i, y_i):i<\kappa\}$ be an enumeration of all existential sentences that hold, where $\theta_i$ is quantifier-free, and the parameters $\bar a_i$ are from $H_n$. 
\begin{claim}
 The set $\Gamma(\{ y_i\}_{i<\kappa})$ consisting of the following formulas is consistent:
 \begin{itemize}
     \item $\{\theta_i(\bar a_i, y_i): i<\kappa\}$;
     \item for each $i<\kappa$, $y_i$ and $H_n$ are $R$-free over $\bar a_i$ $(i<\kappa)$;
     \item $\{y_i\}_{i<\kappa}$ is $R$-free over $H_n$.
\end{itemize}
\end{claim}
By compactness, it is sufficient to show that $\Gamma$ restricted to $y_0,\dots,y_m$ is consistent for $m\in\omega$. 
Let $A \subset H_n$ be a finite subset containing all parameters $\bar a_i$ $(i \leq m)$.
By the free amalgamation property,  $\theta_i$ can be viewed as an $R$-diagram over $A \cup \{y_i\}$ such that $y_i$ and $A$ are $R$-free over $\bar a_i$. 
Choose $b_i$ realizing $\theta_i(A,y_i)$ for each $i$. 
Again, by the free amalgamation property, we can assume that the $b_i$ are $R$-free over $A$. 
Thus, we have established the consistency of all finite subsets of $\Gamma$.;....;/;:;./
(End of Proof of Claim A)

\medbreak
By Claim A, we choose $b_i$ for all $i<\kappa$ realizing $\Gamma$.
Let $H_{n+1}=H_n \cup \{b_i:i<\kappa\}$. Clearly $|H_{n+1}|=\kappa$. 
Notice that each $b_i$ is $R$-adjacent to only finitely many elements in $H_n$. 
Let $F_i$ be this finite set, and choose $m_i \in \omega \setminus f_n(F_i)$. 
Then, $f_{n+1}=f_n \cup \{(i,m_i)\}_{i<\kappa}$ is a valid coloring.
\end{proof}

%% file: o-minimal-graph.tex
We use the cell decomposition theorem for o-minimal structures (see \cite{vddTame}). 
To avoid conflict with the notation for open intervals,
the ordered pair of \(x\) and \(y\) will be denoted by \(\pair{x,y}\) in this subsection.
\begin{theorem}
\label{thm:ominimalcase}
Let \( M = (M, <, \dots) \) be an o-minimal structure.  
Let \( G \subset M \) be a definable set, and let \( R \subset G^2 \) be a definable relation representing the edge set of a graph on \( G \).  
Suppose that \( \chi(G) \) is infinite. Then \( G \) contains arbitrarily large finite cliques.  
In particular, in the monster model, \( G \) contains an infinite clique.
\end{theorem}

\begin{proof}
Since o-minimality is preserved under elementary equivalence, we may assume that $M$, and hence $G$, is sufficiently saturated.
 By o-minimality, the definable set $G$ is a finite union of intervals and points. Because $\chi (G)$ is infinite, at least one of these intervals must have an infinite chromatic number.
Thus, without loss of generality, we may assume that \( G = (d, e) \), where $d,e \in M \cup \{\pm \infty\}$.
From now on, we regard $G$ as a directed graph. Accordingly, its edge set is
$D=\{ \pair{x,y} \in R:y<x\}$. We do not change the definition of coloring. Intuitively, the (directed) edges lie below the line \( y = x \).

We decompose $D$ into finitely many cells. By Remark \ref{rem-first}, among these, there must exist a $\pair{1,1}$-cell $C$ such that the chromatic number of \( (G,C ) \) is infinite.
We choose definable functions $f,g:(d_0, e_0) \to [d,e)$ such that 
    \[
    C=(f,g)=\{\pair{x,y} : d_0<x < e_0, f(x)<y<g(x)\},
    \]
    where $f<g$ holds, and each of $f$ and $g$ is either a strictly monotone definable continuous function or a constant function. 
    Notice also that $g(x) \leq x$ holds. 
\begin{claim}\label{claim-a}
    If there exist infinitely many elements $a \in (d_0,\,e_0)$ with $g(a)=a$, then $G$ contains a definable infinite clique.
\end{claim}
Suppose that there exist infinitely many such elements $a$.
Then, by o-minimality, there exists a non-empty interval 
$I \subset (d_0,\,e_0)$ on which $g(x)=x$ holds. 
Let $a \in I$. Then, we have $f(a)<g(a)=a$. Since $f$ is continuous, there exists a subinterval $(a,\,b) \subset I$ on which $f(x)<a$ holds. 
Then, any two elements $x<y$ from $(a, b)$ are adjacent, since $f(y) <a <x < y=g(y)$ holds. 
(End of the Proof of Claim \ref{claim-a})
\medbreak
Owing to Claim A, we may henceforth assume that there are only finitely many elements a with $g(a)=a$. By restricting to a suitable subinterval of $(d_0,\,e_0)$, we may further assume that
\[\tag{$*$}
g(x)<x\quad \mathrm{for\  all\  }x\in (d_0,\,e_0).
\]

\begin{claim}\label{claim-b}
There exists \( c \in (d_0,\, e_0) \) such that \( g(c) > d_0 \), and the function 
\(
g : (d_0,\, e_0) \to [d, e)
\)
is strictly increasing.
\end{claim}
Suppose, for a contradiction, that 
$
g\big(\,(d_0,\,e_0)\,\big) \leq d_0.
$
Then the possible edges in $G$ are of the form $uv$, where 
\[
u \in (d_0,\,e_0) \quad \text{and} \quad v \in (d,\,d_0).
\]
Hence $G$ is bipartite, and therefore $\chi(G) = 2$, a contradiction.  
Now choose $a \in (d_0,\,e_0)$ with $g(a) > d_0$. 
Notice that $g(a) \in (d_0,\,e_0)$. 
By $(*)$, we have 
\[
g(a) < a \quad \text{and} \quad g(g(a)) < g(a).
\]
From this fact, and noting that $g$ is strictly monotone (or constant), it follows that $g$ must be strictly increasing.
(End of the Proof of Claim \ref{claim-b})
\medbreak

Consider $C_+=\{\pair{x,y} \in C: y>d_0\}$ and $C_-=\{\pair{x,y} \in C: y \leq d_0\}$. $(G,C_-)$ cannot have an infinite chromatic number, since it is bipartite.
Hence, $\chi(G,C_+) \geq \omega$. In 
the following, we restrict our attention to the graph $(G,C_+)$.
Put
\[f^*(x)=\max \{ f(x),d_0\} ,\qquad g^*(x)=\max \{ g(x),d_0\} . \]
By stipulating that $f^*(d_0)=g^*(d_0)=d_0$, both $f^*$ and $g^*$ may be regarded as definable functions from $[d_0,e_0)$ into $[d_0,e_0)$.
We can write \[
C_+=(f^*,\,g^*). 
\]
By Claim \ref{claim-b}, $g^*$ is increasing (in the weak sense). 
We see that $f^*$ is also increasing (in the weak sense). 
Indeed, if $f$ is strictly increasing, clearly $f^*$ is increasing. On the other hand, if $f$ is decreasing, by the same reasoning as in the proof of Claim B, $f(\,(d_0,\,e_0)\,) \subset \, (d,\,d_0]$, and hence $f^*$ is the constant function taking the value $d_0$.

\begin{claim}\label{claim-c}
Let $c \in (d_0,\,e_0)$ be arbitrary.
\begin{enumerate}
    \item The subgraph induced on the interval $[g^*(c),\,c)$ contains no edge. 
    \item Let $m,n \in \omega$ with $|n-m| \geq 2$. 
    Then there is no edge directly connecting the intervals 
    \[
    \big[(f^*)^{m+1}(c), (f^*)^{m}(c)\big)
    \quad \text{and} \quad 
    \big[(f^*)^{n+1}(c), (f^*)^{n}(c)\big) \, .
    \]
\end{enumerate}
\end{claim}
1.     This follows directly from the definition of the upper bound function $g^*$ of $C_+$. Let $u \in (g^*(c),\,c)$.  Since $g^*$ is increasing, $g^*(u) \leq g^*(c)$. Any $v < u$ that is adjacent to $u$ must be less than $g^*(u)$. So, $v$ does not belong to $[g^*(c),c)$. 

2.     This follows from the properties of the lower bound function $f^*$.
(End of the Proof of Claim \ref{claim-c})

\medbreak

\begin{claim}\label{claim-d}
Suppose that there exists an interval $(f^*(c),\,c]$ which contains an infinitely descending sequence 
$\{ (g^*)^n(c) \}_{n \in \omega}$. 
Then the graph $G$ contains an infinite clique.
\end{claim}
By the property of $f^*$ and $g^*$, every point $u \in \, ((g^*)^n(c),\,c]$ is adjacent to all points $v \in \, (f^*(c),(g^*)^{n+1}(c))$. 
Using this observation, we can easily find an infinite clique. 
(End of the Proof of Claim \ref{claim-d})
\medbreak
Owing to Claim \ref{claim-d}, henceforth we assume that no interval of the form $(f^*(c),\,c]$ contains an infinitely descending sequence by $g^*$. 
Moreover, by saturation, there must exist a number $N \in \omega$ such that, for all $c \in (d_0,\,e_0)$, 
\[\tag{$**$}
(g^*)^N(c) \leq f^*(c).
\]
Our goal is to derive a contradiction from this property.
Let us introduce an equivalence relation $\sim$  on $(d_0,\,e_0)$. 
For $a,b \in (d_0,\,e_0)$, we write $a \sim b$, if one of the following holds:
\begin{itemize}
   \item $a \leq b$ and $(f^*)^n(b) \leq a$ for some $n \in \omega$; 
   \item $b \leq a$ and  $(f^*)^n(a) \leq b$ for some $n \in \omega$. 
\end{itemize}
Let \( \bigsqcup_{i < \alpha} X_i \) denote the decomposition of \( (d_0,\, e_0) \) into equivalence classes.
Although each class $X_i$ does not need to be an interval, it is a convex subset of $(d_0,\,e_0)$.  
By the properties of $f^*$ (Claim \ref{claim-c}.2), it follows that there is no edge connecting distinct classes $X_i$ and $X_j$.
Thus, to derive a contradiction, it is sufficient to show the following claim. 

\begin{claim}\label{claim-e}
Let $X_i$ be an arbitrary $\sim$-class.
$X_i$ is colored by using at most $2N$ colors, where $N$ is a number with the property $(**)$. 
\end{claim}
Let \(c \in X_i\). We may assume that \(X_i\) admits a partition into intervals
\[
X_i = \bigcup_{n \in \mathbb{Z}} I_n,
\]
where \(I_n = \left[\, (f^{*})^{n+1}(c),\, (f^{*})^{n}(c) \,\right)\).
First, we show that each $I_n$ can be colored using at most $N$ colors. 
By $(**)$, for some $k < N$, $I_n$ is partitioned into $k+1$ subintervals: 
\[
I_n=\left[f^*(c^*), (g^*)^{k}(c^*)\right) \; \cup \bigcup_{0 \leq i<k}\left[(g^*)^{i+1}(c^*),(g^*)^i(c^*)\right)\, ,
\]
where $c^*=(f^*)^{n}(c)$. 
By Claim C.1, each subinterval above has no edges. 
So, $I_n$ is validly colored by using at most $N$ colors. 
Recall that for all $n \in \Z$, $I_n$, and $I_{n-k}$ with $k \geq 2$ there are no edges that directly connect them. 
So, $X_i$ is $2N$-colored.
(End of the Proof of Claim \ref{claim-e})
\end{proof}
\medbreak

Theorem~\ref{thm:ominimalcase} is optimal in a certain sense.  
Let \(M\) be a structure with a definable total order \(<\).  
Consider the following graph, known as a \emph{shift graph}, which is definable in \(M^{2}\). 
Put
\[
G = \{\pair{u_{1},u_{2}} \in M^{2} : u_{1} < u_{2}\},
\]
and let \(\varphi(u_{1},u_{2},v_{1},v_{2})\) be the formula \(u_{2} = v_{1} \lor u_{1} = v_{2}\).  
Then \(\varphi(u_{1},u_{2};v_{1},v_{2})\) defines an edge relation on \(G\).  
The graph \(G\) is triangle‑free, yet it has an infinite chromatic number \cite{EH1}; the proof is a standard application of Ramsey’s theorem.
For each \(k \ge 2\), one obtains analogous shift graphs definable in \(M^{k}\) using similar formulas.  
These graphs are again triangle‑free and have an infinite (indeed, at least countably infinite) chromatic number.  
The argument is essentially the same as in the case \(k = 2\).